\documentclass[11pt,a4paper]{amsart}

\usepackage{amssymb,enumitem,amsmath,amsthm,bbold,wrapfig}
\usepackage[all]{xy}
\usepackage[in]{fullpage}
\usepackage{color,pxfonts}
\definecolor{Chocolat}{rgb}{0.36, 0.2, 0.09}
\definecolor{BleuTresFonce}{rgb}{0.215, 0.215, 0.36}
\definecolor{EgyptianBlue}{rgb}{0.06, 0.2, 0.65}

\usepackage{graphicx,scalerel}

\usepackage[colorlinks,final,hyperindex]{hyperref}
\hypersetup{citecolor=BleuTresFonce, urlcolor=EgyptianBlue, linkcolor=Chocolat}

\usepackage{fourier}

\newcommand\tsbullet[1][.8]{\mathbin{\ThisStyle{\vcenter{\hbox{%
  \scalebox{#1}{$\SavedStyle\bullet$}}}}}%
}

\catcode`\@=11

\catcode`\@=13

\newtheorem{theorem}{Theorem}

\newtheorem{corollary}{Corollary}
\newtheorem{proposition}{Proposition}

\theoremstyle{definition}
\newtheorem{example}{Example}
\newtheorem*{remark}{Remark}
\newtheorem{definition}{Definition}

\DeclareMathAlphabet{\pazocal}{OMS}{zplm}{m}{n}

\def\calB{\pazocal{B}}
\def\calC{\pazocal{C}}
\def\calD{\pazocal{D}}

\def\calL{\pazocal{L}}
\def\calM{\pazocal{M}}

\def\calO{\pazocal{O}}
\def\calP{\pazocal{P}}
\def\calQ{\pazocal{Q}}
\def\calR{\pazocal{R}}
\def\calS{\pazocal{S}}
\def\calT{\pazocal{T}}
\def\calU{\pazocal{U}}
\def\calV{\pazocal{V}}

\def\calX{\pazocal{X}}

\DeclareMathAlphabet{\mathbbold}{U}{bbold}{m}{n}

\def\kk{\mathbbold{k}}

\DeclareMathOperator{\id}{id}
\DeclareMathOperator{\In}{in}
\DeclareMathOperator{\Hom}{Hom}

\DeclareMathOperator{\PL}{\textsl{PreLie}}
\DeclareMathOperator{\NAP}{\textsl{NAP}}
\DeclareMathOperator{\Lie}{\textsl{Lie}}
\DeclareMathOperator{\Ass}{\textsl{Ass}}

\DeclareMathOperator{\uAss}{\textsl{uAss}}
\DeclareMathOperator{\uCom}{\textsl{uCom}}
\DeclareMathOperator{\RT}{\textsl{RT}}

\DeclareMathOperator{\Dend}{\textsl{Dend}}
\DeclareMathOperator{\Br}{\textsl{Br}}

\DeclareMathOperator{\Tw}{Tw}

\DeclareMathOperator{\Coop}{\textbf{Coop}}

\begin{document}

\title{Enriched pre-Lie operads and freeness theorems}

\author{Vladimir Dotsenko}
\address{Institut de Recherche Math\'ematique Avanc\'ee, UMR 7501, Universit\'e de Strasbourg et CNRS, 7 rue Ren\'e-Descartes, 67000 Strasbourg CEDEX, France}
\email{vdotsenko@unistra.fr}

\author{Lo\"ic Foissy}
\address{LMPA, Centre Universitaire de la Mi-Voix, 50, rue Ferdinand Buisson, CS 80699 -- 62228 Calais CEDEX, France}
\email{foissy@univ-littoral.fr}

\begin{abstract}
In this paper, we study the $\calC$-enriched pre-Lie operad defined by Calaque and Willwacher for any Hopf cooperad $\calC$ to produce conceptual constructions of the operads acting on various deformation complexes. Maps between Hopf cooperads lead to maps between the corresponding enriched pre-Lie operads; we prove criteria for the module action of the domain on the codomain to be free, on the left and on the right. In particular, this implies a new functorial Poincar\'e--Birkhoff--Witt type theorem for universal enveloping brace algebras of pre-Lie algebras.  
\end{abstract}

\keywords{operad, pre-Lie algebra, brace algebra, rooted tree, Poincar\'re--Birkhoff--Witt theorem}
\subjclass[2010]{18D50 (Primary), 05C05, 16D40, 17B35 (Secondary)}

\maketitle

\section{Introduction}

Pre-Lie algebras, also known as right-symmetric algebras, appear in a wide range of research areas from algebra and combinatorics to differential geometry and homotopy theory. One of the early examples of a pre-Lie algebra structure is that on the (shifted) Hochschild cohomology complex of an associative algebra which was famously used by Gerstenhaber \cite{MR161898} to introduce a differential graded Lie algebra structure on that complex. That differential graded Lie algebra controls deformation theory of associative algebras, illustrating an important thesis of contemporary deformation theory: any reasonable deformation problem is controlled by an appropriate differential graded Lie algebra \cite{Lurie,MR2628795}. Hochschild cohomology of an associative algebra has a structure of an algebra over the homology of the little disks operad \cite{getzler1994operads}, prompting the celebrated Deligne conjecture that chains of the operad of little disks act on the Hochschild complex. In many existing proofs of that conjecture \cite{MR1321701,MR1805894,MR1890736,MR1328534}, one uses a remarkable differential graded operad which is often referred to as the brace operad. In fact, the existing terminology is a little bit confusing: Getzler \cite{MR1261901} and Kadeishvili \cite{MR1029003} observed that the pre-Lie algebra structure on the shifted Hochschild complex can be regarded as a part of a bigger structure (still concentrated in homological degree zero) which is also called a brace algebra. Such brace algebras also independently appeared in the work of Ronco on free dendriform algebras \cite{MR1800716}, leading to the Cartier--Milnor--Moore type theorem for dendriform algebras \cite{MR1879927,MR1813766}. 

The two different brace operads are in fact intimately related, and their relationship is best explained by the Willwacher's theory of operadic twisting \cite{MR3299688}: the twisting procedure applied to the brace operad of Getzler--Kadeishvili--Ronco is the differential graded brace operad whose different versions were used in various proofs of the Deligne conjecture. In their recent work on a higher version of Kontsevich's formality theorem, Calaque and Willwacher \cite{MR3411136} generalised that approach and defined, for any Hopf cooperad $\calC$ (not necessarily commutative), the $\calC$-enriched operad $\PL_\calC$. It is well known that elements of the classical pre-Lie operad are linear combinations of labelled rooted trees \cite{MR1827084}, where vertices carry numeric labels coming from numbering the inputs of an operation. In the case of the operad $\PL_\calC$, the elements are $\calC$-enriched labelled rooted trees, that is trees whose vertices carry, in addition to the numeric labels, extra decorations from $\calC$ whose arities match the number of inputs of vertices. This type of decoration does not use elements of $\calC$ as mere labels but rather relies on the structure coming from $\calC$, and so has a clear commonality with the classical notion of enrichment from category theory, hence the terminology. In general, throughout the paper, we use the word ``decorated'' when we talk about labels without extra structure, and ``enriched'' when the extra structure on labels plays a key role in the definitions.

Calaque and Willwacher established that the operad $\PL_\calC$ acts on the underlying space of the deformation complex of any map from $\calO$ to a given operad $\calO'$, where $\calO$ is the Koszul dual operad of the cooperad $\calC$; the corresponding differential graded operad of $\calC$-enriched braces is then obtained by operadic twisting. A version of this operad was implicitly defined within the general theory of  natural operations on deformation complexes proposed by Markl \cite{MR2309979}; papers that utilise that theory in the case of the operad $\Lie$ \cite{MR2325698} and in the case of the associative operad \cite{MR3203365,MR3261598} mention the general construction of braces on cohomology of operadic algebras, but that direction does not seem to have been pursued until recently.   

This paper studies the operad $\PL_\calC$ from the algebraic viewpoint. Our main result concerns the map $\PL_\calB\to \PL_\calC$ arising from a Hopf cooperad map $\calB\to\calC$. We prove criteria relating natural properties of that map to freeness of $\PL_\calC$ as a $\PL_\calB$-module (on the left and on the right). Our arguments rely on proposing an enrichment of another operad based on labelled rooted trees, the operad $\NAP$ introduced by Livernet \cite{MR2244257}. It is known that the composition rule in the operad $\NAP$ may be viewed as retaining the ``leading terms'' of the composition in the operad $\PL$; we show that a version of that statement can be established and used in the enriched context as well. Moreover, we obtain a new conceptual interpretation of the operad $\NAP$ and its generalisations via the left adjoint of the derivative functor from operads to Cauchy monoids. 

There are several motivations for our result. Freeness of the left module structure allows one to prove that free $\PL_\calC$-algebras are free as $\PL_\calB$-algebras, generalising known results like the second author's theorem \cite{MR2724224}. Freeness of the right module can be used in the categorical framework for Poincar\'e--Birkhoff--Witt theorems developed in the first author's joint work with Tamaroff~\cite{dotsenko2018endofunctors}, implying a functorial PBW type theorem for universal enveloping $\PL_\calC$-algebras of $\PL_\calB$-algebras. 
In particular, our results lead to a functorial PBW type theorem for universal enveloping brace algebras of pre-Lie algebras, and, as a byproduct, to another proof of a functorial PBW type theorem for universal enveloping dendriform algebras of pre-Lie algebras first proved in \cite{dotsenko2018endofunctors}. A weaker PBW type theorem for universal enveloping brace algebras of pre-Lie algebras was recently proved by Li, Mo, and Zhao \cite{MR3864250} using Gr\"obner--Shirshov bases. Their methods lead to normal forms in universal enveloping algebras and as such are useful for applications, but our result in particular implies that the PBW isomorphisms can be chosen functorially with respect to algebra morphisms.

This is a short note, and we do not intend to overload it with excessive recollections. All vector spaces in this paper are defined over a field $\kk$ of characteristic zero. 
When writing down elements of operads, we use small latin letters as placeholders; if one works with algebras over operads that carry nontrivial homological degrees, there are extra signs which arise from applying operations to arguments via the usual Koszul sign rule. Readers whose expertise comes from combinatorics are invited to consult the monograph \cite{MR2954392} for intuition on algebraic operads. However, for some specific definitions and notation as well as the viewpoint that emphasises the combinatorics of species, we lean towards the monograph~\cite{MR2724388}. In particular, we use the notation $\times$ for the Hadamard tensor product of species and the notation $\cdot$ for the Cauchy tensor product of species, so that 
\begin{gather*}
(\calP\times\calQ)(I)=\calP(I)\otimes\calQ(I),\\ 
(\calP\cdot\calQ)(I)=\bigoplus_{I=J\sqcup K}\calP(J)\otimes\calQ(K).
\end{gather*}
The singleton species is denoted $\mathbb{X}$. We use the notation $\uCom$ and $\uAss$ for the operads of unital commutative associative algebras and unital associative algebras, respectively. A species whose variations play a central role in this note is that of labelled rooted trees; it satisfies the functional equation
 \[
\RT=\mathbb{X}\cdot\uCom(\RT) . 
 \]
In plain words, this equation means that the datum of a labelled rooted tree consists of the root label and a possibly empty disjoint union of labelled rooted trees on the remaining labels (which is nothing but a $\uCom$-product of labelled rooted trees). 

\section{The monoid enriched operad \texorpdfstring{$\NAP$}{NAP} and its algebraic properties}\label{sec:enrichedRT}

Let $\calS$ be a species. The species $\RT_\calS$ of $\calS$-decorated rooted trees is defined by a functional equation
 \[
\RT_\calS=\mathbb{X}\cdot\calS(\RT) .  
 \] 
Explicitly, we have
 \[
\RT_\calS(I)=\bigoplus_{T\in\RT(I)}\bigotimes_{i\in I}\calS(\In_T(i)) .
 \]
Elements of $\RT_\calS$ are labelled rooted trees where each vertex is additionally decorated by an element of $\calS$ whose arity matches the number of incoming edges of that vertex. It is be convenient to think of these trees as elements of the arity zero component of the free operad generated by $\calS$; later in this note, elements of non-zero arities will also become relevant for one of the key constructions. 

The following definition goes back to works of Mendez and his collaborators \cite{Mendez,MR1239514,MR1087798} which in fact date before the ``renaissance of operads'' in mid 1990s. For brevity, we shall use the term ``Cauchy monoid'' for monoids in the monoidal category of species with respect to the Cauchy tensor product.

\begin{definition}
Let $\calM$ be a Cauchy monoid. We define an operad structure on the species $\RT_\calM$ as follows. The result of insertion of a decorated rooted tree $S$ inside a vertex labelled $i$ of another decorated rooted tree $T$ has the underlying labelled rooted tree where $S$ is grafted in the place of the vertex $i$, and all the subtrees growing from the vertex $i$ in $T$ are grafted at the root of $S$. The decorations of all vertices except for the root of $S$ remain the same, while the label of the root of $S$ becomes equal to the product $ab\in\calM$, where $a$ is the decoration of the vertex $i$ in the decorated tree $T$ and $b$ is the original decoration of the root of the tree $S$. This operad is denoted $\NAP_\calM$ and is called the \emph{$\calM$-enriched non-associative permutative operad}.
\end{definition}

In all examples in the literature that we are aware of, only commutative Cauchy monoids $\calM$ are used. For instance, if $\calM=\uCom$ with its natural Cauchy monoid structure, one obtains the operad $\NAP$ of Livernet~\cite{MR2244257}; the cases of some other commutative monoids $\calM$ are behind the $\NAP$-flavoured operads considered in \cite{MR3302959} and in~\cite{sadi2014weighted}.  However, this construction is valid for all Cauchy monoids, and in some cases produces operads very different from $\NAP$, as we shall see below.

\begin{example}
Let us give an example of a calculation in the operad $\NAP_\calM$. Suppose that $a,b\in\calM(\{1\})$, and let~$bb$ denote the image in $\calB(\{1,2\})$ of the tensor $b\otimes b$ in $\calB(\{1\})\otimes\calB(\{2\})$ under the product map $\calB\cdot\calB\to\calB$. Then we have 
 \[
\vcenter{
\xymatrix@M=3pt@R=5pt@C=5pt{
*+[o][F-]{3,e}\ar@{-}[dr] & & *+[o][F-]{4,e}\ar@{-}[dl]\\
 &  *+[o][F-]{1,bb}\ar@{-}[d]& \\
& *+[o][F-]{2,a} & 
}}
\ = \
\vcenter{
\xymatrix@M=3pt@R=5pt@C=5pt{
*+[o][F-]{3,e}\ar@{-}[d]   \\
*+[o][F-]{*,b}\ar@{-}[d] \\
 *+[o][F-]{2,a} 
}}
\ \circ_* \ 
\vcenter{
\xymatrix@M=3pt@R=5pt@C=5pt{
*+[o][F-]{4,e}\ar@{-}[d]  \\
*+[o][F-]{1,b}
}}
\ = \
\vcenter{
\xymatrix@M=3pt@R=5pt@C=5pt{
*+[o][F-]{*,e}\ar@{-}[d]  \\
*+[o][F-]{2,a}
}}
\ \circ_* \ 
\vcenter{
\xymatrix@M=3pt@R=5pt@C=5pt{
*+[o][F-]{3,e}\ar@{-}[dr] & & *+[o][F-]{4,e}\ar@{-}[dl]\\
 &  *+[o][F-]{1,bb}& \\
}}
 \]
Here $e\in\calB(\varnothing)$ is the unit of $\calB$, and vertex labels like $1,b$ mean that the vertex is labelled by $1$ and additionally decorated by $b$. 
\end{example}

We shall now give a presentation of the operad $\NAP_\calM$ by generators and relations. Recall that a labelled rooted tree $T$ is called a \emph{corolla} if for each of its non-root vertices $v$, we have $\In_\tau(v)=\emptyset$. We define an $\calM$-corolla to be a corolla for which the non-root vertices are decorated by the unit of $\calM$. 
\begin{proposition}\label{prop:NAPgen}
The operad $\NAP_\calM$ is generated by $\calM$-corollas. If we denote by $\langle r;s_1,\ldots,s_m\rangle_\alpha$ the corolla with the root labelled by $r$, the non-root vertices labelled by $s_1,\ldots,s_m$, and the root decoration $\alpha$, all relations in the operad $\NAP_\calM$ follow from the relations
\begin{equation}\label{eq:NAP-C} 
\langle\langle r;s_1,\ldots,s_n\rangle_\alpha;t_1,\ldots,t_m\rangle_\beta=\langle r;s_1,\ldots,s_n,t_1,\ldots,t_m\rangle_{\alpha\beta} .
\end{equation}
\end{proposition}

\begin{proof}
Relations \eqref{eq:NAP-C} follow from the rule for the operad composition. Moreover, these relations already imply that there is a species surjection $\RT_\calM\twoheadrightarrow \NAP_\calM$. Indeed, one can use them to show that the operad $\NAP_\calM$ is spanned by iterated insertions of generators avoiding insertions at the root vertex, and the combinatorics of those iterated insertions is precisely the combinatorics of $\calM$-enriched labelled rooted trees. Consequently, all relations in $\NAP_\calM$ follow from~\eqref{eq:NAP-C}. 
\end{proof}

Relations \eqref{eq:NAP-C} lead to a different interpetation of the operad $\NAP_\calM$. Recall that for an operad~$\calO$, the species derivative $\partial(\calO)$ defined by the formula
 \[
\partial(\calO)(I):=\calO(I\sqcup\{*\}) 
 \]
has a natural structure of a Cauchy monoid (via composition of operations using just the slot $*$), see, for example, \cite[Sec.~3.4.3]{MR3308919}. It turns out that the functor $\NAP$ can be interpreted as a version of the enveloping operad of a graded algebra defined and studied in \cite{MR4023760}; the reader is encouraged to consult \cite{MR3921622} where a general set-up for studying similar functors is established. 

\begin{proposition}
The derivative functor from operads to Cauchy monoids admits a left adjoint, which is given by the functor $\NAP$. 
\end{proposition}

\begin{proof}
We have to prove that 
 \[
\Hom_{\mathrm{monoids}}(\calM,\partial(\calO))\cong 
\Hom_{\mathrm{operads}}(\NAP_\calM,\calO).
 \]
We first note that the Frobenius reciprocity law for group representations implies that for all $n\ge 0$ we have
 \[
\Hom_{S_n}(\calM(n),\mathrm{Res}^{S_{n+1}}_{S_n}(\calO(n+1))\cong 
\Hom_{S_{n+1}}(\mathrm{Ind}^{S_{n+1}}_{S_n}(\calM(n)),\calO(n+1)).
 \]
Since $\mathrm{Res}^{S_{n+1}}_{S_n}(\calO(n+1))\cong\partial(\calO)(n)$ as $S_n$-modules and $\mathrm{Ind}^{S_{n+1}}_{S_n}(\calM(n))\cong (\mathbb{X}\cdot\calM)(n+1)$ as $S_{n+1}$-modules, the Frobenuis reciprocity isomorphisms assemble into an isomorphism 
 \[
\Hom_{\mathrm{species}}(\calM,\partial(\calO))\cong 
\Hom_{\mathrm{species}}(\mathbb{X}\cdot\calM,\calO).
 \]
Homomorphisms of Cauchy monoids are maps of species which are compatible with the products. On the other hand, the species $\mathbb{X}\cdot\calM$ is precisely the species of $\calM$-corollas, and morphisms of operads from $\NAP_\calM$ to $\calO$ are maps of generators that are compatible with the relations between them. Examining the relations \eqref{eq:NAP-C}, we see that the conditions we impose in the two cases coincide.
\end{proof}

Since the composition of left adjoint functors is itself a left adjoint, we arrive at the following result which shows that operads arising from the functor $\NAP$ do not have to exhibit any familiar $\NAP$-type features. 

\begin{corollary}
For a free Cauchy monoid $\calM$, the operad $\NAP_\calM$ is free. 
\end{corollary}

We shall now prove two results on module freeness for the functor $\NAP$. The first of them is completely straightforward, while the second one exhibit interesting unexpected subtleties. We invite the reader to compare the two theorems of this section with \cite[Th.~4]{MR3203367}; while the relationship between these results is not at all direct, they follow the same logic.  

\begin{theorem}\label{th:freeNAPmodule-1}
Let $\calB,\calC$ be two connected Cauchy monoids, and let $\phi\colon\calB\to\calC$ be a Cauchy monoid homomorphism, making $\calC$ a $\calB$-bimodule. We shall consider the left $\NAP_\calB$-module structure on $\NAP_\calC$ arising from the map of operads $\NAP(\phi)\colon\NAP_\calB\to\NAP_\calC$. If the Cauchy monoid $\calC$ is free as a left $\calB$-module, the operad $\NAP_\calC$ is free as a left $\NAP_\calB$-module.
\end{theorem}

\begin{proof}
Let us denote by $\calL$ a species that freely generates $\calC$ as a left $\calB$-module, so that on the level of species we have an isomorphism $\calC\cong\calB\cdot\calL$. Using that isomorphism, we may think of each vertex label of a $\calC$-decorated labelled rooted tree as a combination of labels each of which splits the set of incoming edges into an ordered disjoint union of a set decorated by $\calB$ and a set decorated by $\calL$. Let us take an individual tree $T'$ for which each set of incoming edges of each vertex comes with a splitting like that. Let us consider the maximal connected subgraph $T$ of that tree containing the root whose edges are all decorated by $\calB$. We note that in the operad $\NAP_\calC$ we can write $T'$ as a composition
 \[
T'=\gamma(T;S_1,\ldots,S_m),
 \]
where for each $i$ the set of input edges of the root vertex of the tree $S_i$ is decorated by $\calL$. Since the left $\calB$-module $\calC$ is free, this representation leads to a well defined map of species $\NAP_\calC\to\NAP_\calB\circ\calT_{\calL}^l$, where $\calT_{\calL}^l$ is the species of all $\calC$-decorated rooted trees for which the input edges of the root vertex are decorated by $\calL$. Moreover, the map 
 \[
\NAP_\calB\circ\calT_{\calL}^l\hookrightarrow\NAP_\calC\circ\NAP_\calC\to\NAP_\calC
 \]
obtained from the obvious embedding and the operadic composition is the inverse of the map we constructed, and the resulting isomorphism of species $\NAP_\calC\cong\NAP_\calB\circ\calT_{\calL}^l$ is immediately seen to be a left $\NAP_\calB$-module isomorphism.
\end{proof}

The theorem we just proved makes one wonder whether the same is true for right modules. Originally, we thought that it was the case, and it was not until thorough reading of the anonymous referee that we became convinced that one needs additional assumptions for that. For the reader's convenience, let us offer an example illustrating a problem that may emerge. 

\begin{example}
Consider the Cauchy monoid $\calB$ generated by two elements $x,y\in\calB(1)$ subject to one sole relation $xx=yy$. Furthermore, consider the Cauchy monoid $\calC$ generated by three elements $r,x,y\in\calB(1)$ subject to relations $xx=yy$, $xr=yr=rr=0$. Then there is an obvious map $\calB\to\calC$, and it is immediate to see that $\calC$ is a free right $\calB$-module (generated by the unit and $r$). Let us show that the operad $\NAP_\calC$ is not free as a right $\NAP_\calB$-module. For that, we shall consider the following equalities in the operad $\NAP_\calC$:
$$
\vcenter{
\xymatrix@M=3pt@R=5pt@C=5pt{
*+[o][F-]{4,e}\ar@{-}[d] \\
  *+[o][F-]{2,r}\ar@{-}[d]\\
 *+[o][F-]{*,x} & 
}}
\ \circ_* \ 
\vcenter{
\xymatrix@M=3pt@R=5pt@C=5pt{
*+[o][F-]{3,e}\ar@{-}[d]  \\
*+[o][F-]{1,x} 
}}
\ = \
\vcenter{
\xymatrix@M=3pt@R=5pt@C=5pt{
*+[o][F-]{4,e}\ar@{-}[d] & & \\
  *+[o][F-]{2,r}\ar@{-}[d]& *+[o][F-]{3,e}\ar@{-}[dl]\\
 *+[o][F-]{1,xx} & 
}}
\ = \ 
\vcenter{
\xymatrix@M=3pt@R=5pt@C=5pt{
*+[o][F-]{4,e}\ar@{-}[d] & & \\
  *+[o][F-]{2,r}\ar@{-}[d]& *+[o][F-]{3,e}\ar@{-}[dl]\\
 *+[o][F-]{1,yy} & 
}}
\ = \ 
\vcenter{
\xymatrix@M=3pt@R=5pt@C=5pt{
*+[o][F-]{4,e}\ar@{-}[d]   \\
*+[o][F-]{2,r}\ar@{-}[d] \\
 *+[o][F-]{*,y} 
}}
\ \circ_* \ 
\vcenter{
\xymatrix@M=3pt@R=5pt@C=5pt{
*+[o][F-]{3,e}\ar@{-}[d]  \\
*+[o][F-]{1,y} 
}} .
$$
We observe that the equality of the element on the left and the element on the right is a nontrivial relation in the right $\NAP_\calB$-module $\NAP_\calC$: the presence of the decoration $r$ implies that the elements $\vcenter{
\xymatrix@M=3pt@R=5pt@C=5pt{
*+[o][F-]{4,e}\ar@{-}[d] \\
  *+[o][F-]{2,r}\ar@{-}[d]\\
 *+[o][F-]{*,x} & 
}}$ and $\vcenter{
\xymatrix@M=3pt@R=5pt@C=5pt{
*+[o][F-]{4,e}\ar@{-}[d]   \\
*+[o][F-]{2,r}\ar@{-}[d] \\
 *+[o][F-]{*,y} 
}}$ are indecomposable elements of that module. 
\end{example}

This example means that we need to impose some constraints, and in fact, gives a good hint as to what constraint to impose. The final result is as follows. 

\begin{theorem}\label{th:freeNAPmodule-2}
Let $\calB,\calC$ be two connected Cauchy monoids, and let $\phi\colon\calB\to\calC$ be a Cauchy monoid homomorphism, making $\calC$ a $\calB$-bimodule. We shall consider the right $\NAP_\calB$-module structure on $\NAP_\calC$ arising from the map of operads $\NAP(\phi)\colon\NAP_\calB\to\NAP_\calC$.
Suppose that one of the following conditions holds:
\begin{itemize}
\item the Cauchy monoid $\calB$ is free,
\item the Cauchy monoid $\calB$ is free commutative.
\end{itemize}
Then the operad $\NAP_\calC$ is free as a right $\NAP_\calB$-module.
\end{theorem}

\begin{proof}
Let us denote by $\calR\subset\calC$ a species that freely generates $\calC$ as a right $\calB$-module, so that on the level of species we have $\calC\cong\calR\cdot\calB$. Using that isomorphism, each $\calC$-decorated labelled rooted tree can be written as a combination of trees for which each label of each internal vertex is a product $x\cdot y\in\calR(J)\otimes \calB(K)\subset\calR\cdot\calB$; here $J\sqcup K$ is the set of the input edges of that internal vertex. Moreover, because of our hypothesis on $\calB$, we may assume that $y=y_1\cdots y_p\in \calB(K_1)\otimes\cdots\otimes\calB(K_p)$ is a product of \emph{generators} of $\calB$ (either defined uniquely or uniquely up to permutation of factors). For an individual tree $T'$ like that, we shall refer to the edges from $J$ as $r$-edges and to the edges from $K$ as $b$-edges. 

For the following definition, we shall fix a vertex $j$ of our tree $T'$. We shall call $j$ important if the maximal subtree of $T'$ rooted at $j$ contains at least one $r$-edge. For an important vertex $j$, let $x\cdot y_1\cdots y_p\in \calR(J)\otimes\calB(K_1)\otimes\cdots\otimes\calB(K_p)$ be the decomposition of its label. We shall call the $\calB$-factor~$y_i$ very good if for each $b$-edge $e$ from $K_i$, all of its descendant edges (edges $e'$ for which the path to the root from $e'$ passes via $e$) are $b$-edges as well. Very good factors more or less determine what can be factored out as the right action of the operad $\NAP_\calB$: it is at this point that it becomes important whether $\calB$ is free as a Cauchy monoid or as a commutative Cauchy monoid, and we shall now explain how to deal with both cases.   

\textsl{Case 1. } If the Cauchy monoid $\calB$ is free, all the ``terminal'' very good $\calB$-factors of important vertices can be factored out. To make this precise, for each important vertex $j$, we define the subtree $T_j$ of $T'$ as follows: it is rooted at $j$ and includes all descendants of $j$ for which the path to the root passes via an edge from~$K_s$, for each $s$ such that the factors $y_s$, \ldots, $y_p$ are all very good. (In particular, for some vertices $j$, the subtree $T_j$ consists just of the vertex $j$: this means that the last $\calB$-factor $y_p$ is not very good.) Then in the operad $\NAP_\calC$ we can write any tree $T'$ with the set of important vertices $\{j_1, \ldots, j_m\}$ as a composition
 \[
T'=\gamma(S;T_{j_1},\ldots,T_{j_m}),
 \]
where the tree $S$ has $m$ vertices, and the last $\calB$-factor of each of them, if exists, is not very good. Such trees $S$ span a species that we denote $\calT_{\calR}^r$. 

\textsl{Case 2. } If the Cauchy monoid $\calB$ is free commutative, all the very good factors of important vertices can be factored out. Namely, for each important vertex $j$, we define the subtree $T_j$ of $T'$ as follows: it is rooted at $j$ and includes all descendants of $j$ for which the path to the root passes via an edge from~$K_s$, for each $s$ such that the factor $y_s$ is very good. (In particular, for some vertices $j$, the subtree~$T_j$ consists just of the vertex $j$: this means that this vertex has no very good $\calB$-factors.) Then in the operad $\NAP_\calC$ we can write any tree $T'$ with the set of important vertices $\{j_1, \ldots, j_m\}$ as a composition
 \[
T'=\gamma(S;T_{j_1},\ldots,T_{j_m}),
 \]
where the tree $S$ has $m$ vertices, and each of them has no very good $\calB$-factors. Such trees $S$ span a species that we denote $\calT_{\calR}^r$. 

Combining the freeness of the right $\calB$-module $\calC$ with the freeness of the Cauchy monoid $\calB$, in each of these two cases our composition formula leads to a well defined map of species $\NAP_\calC\to\calT_\calR^r\circ \NAP_\calB$. Moreover, the map 
 \[
\calT_\calR^r\circ \NAP_\calB\hookrightarrow\NAP_\calC\circ\NAP_\calC\to\NAP_\calC
 \]
obtained from the obvious embedding and the operadic composition is the inverse of the map we constructed, and the resulting isomorphism of species $\NAP_\calC\cong\calT_\calR^r\circ \NAP_\calB$ is immediately seen to be a right $\NAP_\calB$-module isomorphism.
\end{proof}

\section{The Hopf cooperad enriched operad \texorpdfstring{$\PL$}{PL} and its algebraic properties}\label{sec:enrPL}

The operad $\NAP$ is, in a sense, a degeneration of a much more interesting operad on the linearisation of the species of rooted trees, the pre-Lie operad. Let us recall the construction of that operad due to Chapoton--Livernet~\cite{MR1827084}. The underlying species of the operad $\PL$ is also the species $\RT$ of labelled rooted trees, but the insertion of a labelled rooted tree $S$ at a vertex $i$ of a labelled rooted tree~$T$ is equal to the sum
 \[
\sum_{f\colon \In_T(i)\to J} T\circ_i^f S ,
 \]
where the sum is over all functions $f$ from the set of incoming edges of the vertex labelled $i$ to the set~$J$ of vertices of $S$; the labelled rooted tree $T\circ_i^f S$ is obtained by grafting the tree $S$ in the place of the vertex $i$, and grafting the subtrees growing from the vertex $i$ in $T$ at the vertices of $S$ according to the function $f$, so that the set of incoming edges of each vertex $j$ becomes $\In_S(j)\sqcup f^{-1}(j)$. Of course, the $\NAP$ insertion corresponds to the function $f$ for which $f^{-1}(j)=\emptyset$ for all $j$ different from the root. 

We are not aware of a way to generalise this construction to an operad structure on $\RT_\calM$ where~$\calM$ is an arbitrary Cauchy monoid; it turns out that the right structure on the species of decorations is that of a Hopf cooperad. The corresponding definition was originally given by  Calaque and Willwacher \cite[Sec.~3.1.2]{MR3411136}; we spell it out in detail to ensure consistency with our terminology and notation.

We feel that it would be beneficial to the reader to have a reminder of precise definitions related to operads and to Hopf cooperads so that there is no confusion among the existing variations of that notion (for example, the references \cite{MR3616816,khoroshkin2017real} assume Hopf cooperads commutative while the references \cite{MR2724388,getzler1994operads} do not). 

\begin{definition}
A \emph{cooperad} is a coassociative comonoid in the category of species equipped with the operation $\circ'$ defined by
\begin{gather*}
\calP\circ'\calQ=\prod_n \left(\calP(n)\otimes \calQ^{\cdot n}\right)^{S_n}.
\end{gather*}
\end{definition}
We remark that for each cooperad $\calC$, we may use the structure map $\Delta\colon\calC\to\calC\circ'\calC$ followed by appropriate counit maps $\calC\twoheadrightarrow\calC(1)\to\kk$ to get a cooperad with respect to the definition involving ``composition coproducts'' \cite{MR3616816}. It is well known that for each tree $T$ with the set of leaves $I$ there is a map $\Delta_T\colon\calC(I)\to\calT^c(\calC)(I)$ from the component of the cooperad $\calC$ with the indexing set $I$ to the component of the cofree cooperad on $\calC$ with the same indexing set; we shall refer to this map as the decomposition map according to the tree $T$. We refer the reader to \cite[Th.~9.1.9 and Sec.~C.1]{MR3616816} for a discussion of this construction in the case of $\calC(0)=\{0\}$ and $\calC(1)\cong\kk$; if one merely needs to construct the map, and not establish an equivalence of several different definitions, these assumptions are not necessary, and the map $\Delta_T$ is readily available. 

\begin{definition}
A \emph{Hopf cooperad} is a monoid in the symmetric monoidal category $(\Coop,\times,\uCom^*)$ arising from the ``standard'' (factor-wise) structure of a cooperad on the Hadamard product $\calC_1\times\calC_2$ of two cooperads. In plain words, a Hopf cooperad is a cooperad $\calC$ equipped with an associative product $\mu\colon \calC\times\calC\to \calC$ and a unit map $\eta\colon\uCom^*\to\calC$ which are morphisms of cooperads and satisfy the usual axioms of the product and the unit in a monoid. A Hopf cooperad $\calC$ is said to be \emph{connected} if the map $\eta_0\colon \kk=\uCom^*(0)\to \calC(0)$ is an isomorphism.
\end{definition}

The next result is essentially a dual of \cite[Th.~2.3.3]{MR2423811}; because of its importance for our arguments, we give a complete proof. A particular case of this result is also implicit in~\cite{MR3411136}.

\begin{proposition}\label{prop:monoid}
Let $\calC$ be a connected Hopf cooperad. Then the underlying species of $\calC$ can be given a structure of a Cauchy monoid (denoted by~$\calC^{\tsbullet}$). This Cauchy monoid structure depends only on the Hopf cooperad structure on $\calC$, and therefore it is functorial with respect to maps of cooperads. 
\end{proposition}

\begin{proof}
Because of the connectedness assumption, the composite 
\begin{equation}\label{eq:unit}
\calC(n)\to(\calC\circ'\calC)(n)\to(\calC(n+m)\otimes(\calC(1)^{\otimes n}\otimes\calC(0)^{\otimes m}))^{S_m}\to (\calC(n+m)\otimes(\kk^{\otimes n}\otimes\calC(0)^{\otimes m}))^{S_m}
\end{equation}
of the full cooperad decomposition map, the projection on the appropriate summand of $\calC\circ'\calC$ (the summand of $\left(\calC(n+m)\otimes \calC^{\cdot (n+m)}\right)^{S_{n+m}}$ where we take the terms of $\calC^{\cdot (n+m)}(n)$ corresponding to partitions into $n$ singletons, and rearrange terms using the symmetric group actions using the standard identification of invariants with coinvariants), and the cooperad counit $\calC(1)\to\kk$ can be viewed as a map $\calC(n)\to\calC(n+m)^{S_m}\subset\calC(n+m)$. Consequently, for all $I,J$ we have a sequence of maps
\begin{equation}\label{eq:prodIJ}
\calC(I)\otimes\calC(J)\to\calC(I\sqcup J)\otimes\calC(I\sqcup J)\to \calC(I\sqcup J),
\end{equation}
where the last arrow is simply the product in the algebra $\calC(I\sqcup J)$; the datum of all such maps is precisely a map $\nu\colon\calC\cdot\calC\to\calC$. The associativity of $\nu$ follow from the fact that $\mu$ is a morphism of cooperads, from the associativity of the product $\mu$, and from the coassociativity and counitality of the cooperad decomposition maps. Let us show that the element $1\in\kk=\calC(0)$ is the unit of the associative product~$\nu$. We note that since the unit map $\eta$ is a morphism of cooperads, the image of the composite
\begin{equation}\label{eq:produnit}
\kk=\calC(0)\to(\calC\circ'\calC)(0)\to(\calC(n)\otimes \calC(0)^{\otimes n})^{S_n}
\end{equation}
of the full cooperad decomposition map and the projection on the appropriate component of $\calC\circ'\calC$ sends the basis element $1\in\kk=\calC(0)$ to the $\eta_n(1)\otimes1^{\otimes n}$, where $\eta_n(1)$ is the image of the basis element $1\in\kk=\uCom^*(n)$ under the unit map $\eta\colon\uCom^*\to\calC$. Since $\eta_n(1)$ is precisely the unit of the associative algebra $\calC(n)$, this proves the unit axiom for the product $\nu$. Consequently, $\calC$ acquires a Cauchy monoid structure.
\end{proof}

We are now ready to define the protagonist of this paper, the operad $\PL_\calC$.

\begin{definition}
Let $\calC$ be a connected Hopf cooperad. We shall define an operad structure on $\RT_\calC$ as follows. Let $S,T$ be two $\calC$-decorated rooted trees. The insertion operation $T\circ_i S$, where $i$ is a vertex of $T$, is defined by the formula 
 \[
\sum_{f\colon \In_T(i)\to J} T\widetilde{\circ}_i^f S 
 \]
generalising the Chapoton--Livernet formula, which we shall now describe. First, the underlying (non-decorated) rooted tree of $T\widetilde{\circ}_i^f S$ is given by the operation $\circ_i^f$ applied to the underlying rooted trees of $T$ and $S$. The decorations of vertices of this tree are as follows. For each vertex coming from the tree $T$ (except for the vertex $i$), its decoration is equal to its decoration in the tree $T$. For decorations of vertices coming from the tree $S$, one has to invoke the cooperad structure of $\calC$. Note that the operation $\circ_i^f$ changes the sets of incoming edges for vertices coming from the tree $S$: the set of incoming edges of each vertex $j$ becomes $\In_S(j)\sqcup f^{-1}(j)$. We shall define an auxiliary tree $S_{i,f}$ which is obtained from the underlying rooted tree of $S$ by creating at each its vertex labelled $j$ new half-edges (leaves) indexed by the set $f^{-1}(j)$. The full set of leaves of the tree $S_{i,f}$ is $\In_T(i)$, therefore one may apply the cooperad decomposition map $\Delta_{S_{i,f}}$ to an element $c\in \calC(\In_T(i))$; since in this tree the set of incoming edges of each vertex $j$ is $\In_S(j)\sqcup f^{-1}(j)$, the decomposition map gives an element from 
\begin{equation}\label{eq:decomp}
\bigotimes_{j\in J}\calC(\In_S(j)\sqcup f^{-1}(j)).
\end{equation} 
At the same time, original decoration of the tree $S$ belongs to
\begin{equation}
\bigotimes_{j\in J}\calC(\In_S(j)).
\end{equation}
To obtain the decoration in $T\widetilde{\circ}_i^f S$ of each vertex $j$ coming from the tree $S$, one has to compute the product
 \[
\calC(\In_\sigma(j)\sqcup f^{-1}(j))\otimes \calC(\In_\sigma(j))\to \calC(\In_\sigma(j)\sqcup f^{-1}(j))\otimes \calC(\In_\sigma(j)\sqcup f^{-1}(j))\to\calC(\In_\sigma(j)\sqcup f^{-1}(j))
 \] 
(defined analogously to the sequence of maps \eqref{eq:prodIJ}) of the decoration arising from the decoration of the vertex $i$ under the decomposition map and the decoration of the vertex $j$ in $S$.
The collection of all insertion operations $T\circ_i S$ makes the collection of $\calC$-decorated rooted trees an operad which is denoted $\PL_\calC$ and called \emph{the $\calC$-enriched pre-Lie operad}.
The construction $\calC\mapsto\PL_\calC$ (creating the $\calC$-enriched pre-Lie operad from a Hopf cooperad $\calC$) is functorial in $\calC$: if $\phi\colon\calC\to\calD$ is a map of connected Hopf cooperads, there is an induced map  
 \[
\PL(\phi)\colon\PL_{\calC}\to\PL_\calD. 
 \]
\end{definition}

\begin{remark}\leavevmode
\begin{enumerate}
\item If $\calC$ is a usual associative and coassociative bialgebra, one can regard it as a Hopf cooperad supported at arity one. One can extend it in an obvious way by an element $1$ of arity zero; the pre-Lie operad constructed of this Hopf cooperad can be interpreted as a linearised version of the ``word operad'' $\mathbb{W}_\calC$ \cite{MR4114993,MR3306080}. 
\item For the case $\calC=\uCom^*$, the decorations of vertices 
are ``trivial'' (each decoration is determined by the number of input edges), and one obtains the pre-Lie operad $\PL$ itself.
\item For the case $\calC=\uAss^*$, decorating each vertex of a tree with an associative (co)operation indexed by the inputs is the same as considering planar rooted trees. Moreover, the way decorations are used in the definition above in fact leads to the classical construction of the brace operad via substitutions of planar rooted trees~\cite{MR1879927,MR1909461}.
\end{enumerate}
\end{remark}

As we mentioned above, the sum defining insertion formula in the operad $\PL$ includes the term describing the tree insertion in the operad $\NAP$; this allows to utilise the operad $\NAP$ as a technical tool in results about the pre-Lie operad \cite{MR2586994,MR2724224}. We shall now see that the same is true for the operad $\PL_\calC$, so that the operad structure of $\PL_\calC$ ``deforms'' the operad structure of $\NAP_\calC$ by adding lower terms. We believe that one can make the word ``deforms'' precise using a formalism similar to that of \cite{MR3096857,sadi2014weighted}, but it is not going to play a role in our arguments which rely on the following elementary combinatorial observation.

\begin{proposition}
Let us consider for each tree $T\in\RT(I)$, the induced partial order on $I$. For any $i\in I$, and for any tree $S\in\RT(J)$, the Chapoton--Livernet composition $T\circ_i S$ in the operad $\PL$ is the sum of all trees whose partial order on $I
\circ_i J$ refines the order obtained from the orders on $I$ and on $J$ by identification of the root vertex of $S$ with $i$ and whose restrictions to $I$ and to $J$ coincide with the partial orders prescribed by $S$ and by $T$ respectively. 
\end{proposition} 

\begin{proof}
This is an immediate reformulation of the definition.
\end{proof}

This result allows us to refer to the term $T\circ_i^f S$ corresponding to the function $f$ for which $f^{-1}(j)=\emptyset$ for all $j$ different from the root, that is the insertion in the operad $\NAP$, as the \emph{leading term} for the composition in $\PL$. Similarly, we refer to the term $T\widetilde{\circ}_i^f S$ corresponding to the function $f$ for which $f^{-1}(j)=\emptyset$ for all $j$ different from the root as the \emph{leading term} for the composition in $\PL_\calC$. We shall now see that this leading term is given by the composition in the operad $\NAP_{\calC^{\tsbullet}}$ associated to the Cauchy monoid $\calC^{\tsbullet}$ from Proposition \ref{prop:monoid}.

\begin{proposition}\label{prop:LeadingTerms}
In the law for the insertion operation $T\circ_i S$ in the operad $\PL_\calC$, the leading term is precisely the insertion of $S$ at the vertex $i$ of $S$ in the operad $\NAP_{\calC^{\tsbullet}}$.
\end{proposition}

\begin{proof}
By definition, the underlying labelled rooted tree of $T\circ_i S$ is obtained by the insertion in the operad $\NAP$. Let us examine the $\calC$-decoration of that tree. According to the general rule, one must compute the cooperad decomposition of the decoration of $i$ corresponding to certain tree $\Gamma$. That tree $\Gamma$ is obtained from the labelled rooted tree of $S$ by adding at its root vertex $r$ extra incoming half-edges that are indexed by the set $\In_T(i)$. As an example, for a concrete insertion $T\circ_i S$ we have
 \[
T=\vcenter{\hbox{\xymatrix@M=5pt@R=10pt@C=10pt{
*+[o][F-]{k}\ar@{-}[d] & \\
*+[o][F-]{i}\ar@{-}[d] &\\
*+[o][F-]{j}  
}}} ,\qquad 
S=
\vcenter{\hbox{
\xymatrix@M=5pt@R=10pt@C=10pt{
*+[o][F-]{t}\ar@{-}[dr] & & *+[o][F-]{u}\ar@{-}[dl]\\
 &  *+[o][F-]{s}\ar@{-}[d]& \\
& *+[o][F-]{r} & 
}}}
, \qquad \Gamma= \vcenter{\hbox{\xymatrix@M=5pt@R=10pt@C=10pt{
*+[o][F-]{t}\ar@{-}[dr] & & *+[o][F-]{u}\ar@{-}[dl]\\
k\ar@{-}[dr] &  *+[o][F-]{s}\ar@{-}[d]& \\
& *+[o][F-]{r} & 
}}}
 \]
We note that cooperad decomposition map corresponding to this tree is obtained in two steps. The first is the decomposition 
 \[
\calC(\In_T(i))\to\calC(\In_T(i)\sqcup\In_S(r))
 \]
which is a particular case of the maps \eqref{eq:unit} used to define the Cauchy monoid structure, and the second is made of maps \eqref{eq:produnit} which reproduce the decomposition maps in the cooperad $\uCom^*$. As a consequence, multiplying the decorations of vertices of $S$ by the decorations obtained by applying this decomposition map to the decoration of the vertex $i$ in $T$ simply multiplies the label of the root vertex of $S$ by the label of $i$ on the left; all other decorations are elements $\eta_j(1)$ which are the units of the corresponding algebras. Therefore, we recover the operad structure of $\NAP_{\calC^{\tsbullet}}$. 
\end{proof}

This calculation has one important implication. 

\begin{corollary}\label{cor:PLgen}
The operad $\PL_\calC$ is generated by $\calC$-corollas. 
\end{corollary}

\begin{proof}
This follows from Proposition \ref{prop:NAPgen} by an inductive argument using refinement of partial orders.
\end{proof}

We shall now prove the following result which is the main theorem of this paper. 

\begin{theorem}\label{th:freePLmodule}
Let $\calB,\calC$ be connected Hopf cooperads, and let $\phi\colon\calB\to\calC$ be a map of Hopf cooperads. We consider the corresponding map of Cauchy monoids $\phi^{\tsbullet}\colon\calB^{\tsbullet}\to\calC^{\tsbullet}$, the $\calB^{\tsbullet}$-bimodule structure on the Cauchy monoid~$\calC^{\tsbullet}$ defined using the map~$\phi^{\tsbullet}$, and the $\PL_\calB$-bimodule structure on $\PL_\calC$ arising from the map of operads $\PL(\phi)\colon\PL_\calB\to\PL_\calC$.
\begin{enumerate}
\item If the Cauchy monoid $\calC^{\tsbullet}$ is free as a left $\calB^{\tsbullet}$-module, the operad $\PL^\calC$ is free as a left $\PL^\calB$-module.
\item If the Cauchy monoid $\calB^{\tsbullet}$ is free commutative and the Cauchy monoid $\calC^{\tsbullet}$ is free as a right $\calB^{\tsbullet}$-module, the operad $\PL^\calC$ is free as a right $\PL^\calB$-module. 
\end{enumerate}
\end{theorem}

\begin{proof}
From Theorems \ref{th:freeNAPmodule-1} and \ref{th:freeNAPmodule-2}, we already know that
\begin{enumerate}
\item if the Cauchy monoid $\calC^{\tsbullet}$ is free as a left $\calB^{\tsbullet}$-module with the species of generators $\calL$, the operad $\NAP_{\calC^{\tsbullet}}$ is free as a left $\NAP_{\calB^{\tsbullet}}$-module with the species of generators $\calT_\calL^l$ of all $\calC$-decorated rooted trees for which the input edges of the root vertex are decorated by~$\calL$,
\item if the Cauchy monoid $\calB^{\tsbullet}$ is free commutative and the Cauchy monoid $\calC^{\tsbullet}$ is free as a right $\calB^{\tsbullet}$-module with the species of generators $\calR$, the operad $\NAP_{\calC^{\tsbullet}}$ is free as a right $\NAP_{\calB^{\tsbullet}}$-module with the species of generators $\calT_{\calR}^r$ spanned by $\calC$-decorated rooted trees for which the vertices have no very good $\calB$-factors.
\end{enumerate}

We claim that the same species of generators work in each of the two cases, that is
\begin{enumerate}
\item if the Cauchy monoid $\calC^{\tsbullet}$ is free as a left $\calB^{\tsbullet}$-module with the species of generators $\calL$, the operad $\PL_{\calC}$ is free as a left $\PL_{\calB}$-module with the species of generators $\calT_\calL^l$ of all $\calC$-decorated rooted trees for which the input edges of the root vertex are decorated by~$\calL$,
\item if the Cauchy monoid $\calB^{\tsbullet}$ is free commutative and the Cauchy monoid $\calC^{\tsbullet}$ is free as a right $\calB^{\tsbullet}$-module with the species of generators $\calR$, the operad $\PL_{\calC}$ is free as a right $\PL_{\calB}$-module with the species of generators $\calT_{\calR}^r$ spanned by $\calC$-decorated rooted trees for which the vertices have no very good $\calB$-factors.
\end{enumerate}

To establish that, we shall use Proposition \ref{prop:LeadingTerms}. Indeed, since the compositions in the operad $\NAP_{\calC^{\tsbullet}}$ are the leading terms of the compositions in the operad $\PL_\calC$, each linear independence in the $\PL_\calB$-module $\PL_\calC$ follows from the same linear independence in the $\NAP_{\calB^{\tsbullet}}$-module $\NAP_{\calC^{\tsbullet}}$ as the coefficients in composition change by an upper triangular matrix corresponding to refinement of partial orders.
\end{proof}

We remark that for any Hopf cooperad $\calB$, the Cauchy monoid $\calB^{\tsbullet}$ has a generator $c$ of arity one corresponding to a section of the counit map; that element $c$ generates a non-free Cauchy submonoid $\uCom^*$, and hence the Cauchy monoid $\calB^{\tsbullet}$ cannot be free. Thus, the first possibility of Theorem \ref{th:freeNAPmodule-2} cannot occur in this case, and one may focus on the case of a free commutative Cauchy monoid. 

\section{Applications and further directions}

\subsection{Verification of the freeness condition for the unit map}

Let us indicate two situations when $\calC$ is free as a $\calB$-module in the particular case $\calB=\uCom^*$ (and the map $\phi\colon\uCom^*\to\calC$ equal to the unit of the Hopf cooperad $\calC$).

\begin{proposition}\label{prop:twocases}
Suppose that one of the following conditions holds: 
\begin{enumerate}
\item the Hopf cooperad $\calC$ is augmented, i.e. there is a map of Hopf cooperads $\epsilon\colon\calC\to\uCom^*$ such that $\epsilon\eta=\id$. 
\item components of the cooperad $\calC$ are finite-dimensional, the corresponding Cauchy monoid $\calC^{\tsbullet}$ is commutative, and $\calC$ is a Hopf cooperad with comultiplication, i.e. there is a map of Hopf cooperads $\nu\colon \calC\to\uAss^*$.
\end{enumerate}
Then the $(\uCom^*)^{\tsbullet}$-module action on $\calC^{\tsbullet}$ via the unit map is free (on the left and on the right).
\end{proposition}

\begin{proof}
To establish this result in the augmented case, we note that the composite  
 \[
\calC\to\calC\circ'\calC\to\uCom^*\circ'\calC
 \]
made of the cooperad structure and the augmentation clearly defines a $\uCom^*$-coalgebra structure on $\calC$. Moreover, since $\calC$ is a Hopf cooperad, the Cauchy monoid structure on $\calC$ and the thus defined $\uCom^*$-coalgebra satisfy the Hopf compatibility relation in the symmetric monoidal category of species with respect to the Cauchy tensor product. From the running connectedness assumption, it follows that the Cauchy monoid $\calC^{\tsbullet}$ is the twisted universal enveloping algebra of the twisted Lie algebra of primitive elements \cite{MR2504663,MR1218107}, and the Cauchy monoid $(\uCom^*)^{\tsbullet}$ is its subalgebra which is the universal envelope of the one-dimensional Lie subalgebra spanned by the singleton species. From the analogue of the theorem of Poincar\'e--Birkhoff--Witt for twisted universal envelopes, it follows that $\calC^{\tsbullet}$ is free as a $(\uCom^*)^{\tsbullet}$-module, both on the left and on the right.  

To establish the result in the case of a cooperad with comultiplication, one starts in the similar way and obtains a map  
 \[
\calC\to\calC\circ'\calC\to\uAss^*\circ'\calC
 \]
which may be used to define a $\uAss^*$-coalgebra structure on $\calC$. Moreover, we assumed the Cauchy monoid~$\calC$ to be commutative, so we have a commutative Cauchy monoid structure and the coassociative coalgebra structure related by the Hopf compatibility relation. Thus, dualising the previous argument, we observe that $\calC^{\tsbullet}$ is a free commutative Cauchy monoid, and hence a free module (both on the left and on the right) over the Cauchy submonoid generated by the singleton species, which is exactly $(\uCom^*)^{\tsbullet}$.
\end{proof}

One important instance where the first situation described by the proposition applies is the case of a Hopf cooperad obtained as cohomology cooperad of a topological operad made of connected spaces; in this case the augmentation is the map that kills all elements of positive homological degree. An instance of the second situation is the case of the cooperad $\uAss^*$ itself, which we shall now discuss in detail. 

\subsection{The ``classical'' brace operad}

As we mentioned above, the operad $\PL_{\uAss^*}$ is the operad $\Br$ whose algebras are classical brace algebras of \cite{MR1261901,MR1029003,MR1800716}. The unit map $\uCom^*\to\uAss^*$ leads to an operad map 
 \[
\PL=\PL_{\uCom^*}\to\PL_{\uAss^*}=\Br, 
 \]
which was previously constructed directly in \cite{MR2287127}. 

\begin{theorem}\label{th:brace}
The brace operad $\Br$ is free as a left $\PL$-module and as a right $\PL$-module.
\end{theorem}

\begin{proof}
According to the second part of Proposition \ref{prop:twocases}, the Cauchy monoid $(\uAss^*)^{\tsbullet}$ is a free module (on either side) over the free commutative Cauchy monoid $(\uCom^*)^{\tsbullet}$, so Theorem \ref{th:freePLmodule} applies. 
\end{proof}

This result has two immediate consequences that we record below. 

\begin{corollary}\leavevmode
\begin{enumerate}
\item Free brace algebras are free when considered as pre-Lie algebras \cite{MR2724224}.
\item There exists an analytic endofunctor $\calU$ such that the underlying vector space of the universal enveloping brace algebra of a pre-Lie algebra $L$ is isomorphic to $\calU(L)$ functorially with respect to pre-Lie algebra maps. 
\end{enumerate}
\end{corollary}

\begin{proof}
The first statement is immediate from the freeness as a left module: if we have a left $\PL$-module isomoprphism $\Br\cong\PL\circ\calT_\calL^l$, the free brace algebra $\Br(V)$ is isomorphic to the free pre-Lie algebra generated by $\calT_\calL^l(V)$.

The second statement follows from the freeness as a right module: if we have a right $\PL$-module isomorphism $\Br\cong\calT_\calR^r\circ\PL$, the result of \cite[Th.~3.1]{dotsenko2018endofunctors} implies that the endofunctor $\calU:=\calT_\calR$ works: the underlying vector space of the universal enveloping brace algebra of a pre-Lie algebra $L$ is isomorphic to $\calT_\calR^r(L)$ functorially with respect to pre-Lie algebra maps.
\end{proof}

It turns out that the second of those results can be immediately used to give a new proof of the following statement that was first proved in~\cite[Th.~4.6]{dotsenko2018endofunctors}

\begin{corollary}
There exists an analytic endofunctor $\calV$ such that the underlying vector space of the universal enveloping dendriform algebra of a pre-Lie algebra $L$ is isomorphic to $\calV(L)$ functorially with respect to pre-Lie algebra maps.
\end{corollary}

\begin{proof}
From earlier work of Chapoton \cite{MR1879927} and Ronco \cite{MR1813766}, it follows that the operad of dendriform algebras $\Dend$ is a free right $\Br$-module; in fact, the ``natural'' space of generators of that module is $\Ass^{*}$: each free dendriform algebra $\Dend(V)$ has a structure of a cofree conilpotent coassociative coalgebra, and the space of cogenerators of that coalgebra is precisely $\Br(V)$. Since we just proved that there is a right $\PL$-module isomorphism $\Br\cong\calT_\calR^r\circ\PL$, we have a right $\PL$-module isomorphism 
 \[
\Dend\cong\Ass^*\circ\Br\cong\Ass^*\circ\calT_\calR^r\circ\PL,
 \]
and therefore the result of \cite[Th.~3.1]{dotsenko2018endofunctors} implies that the endofunctor $\calV:=\Ass^*(\calT_\calR^r)$ works: the underlying vector space of the universal enveloping dendriform algebra of a pre-Lie algebra $L$ is isomorphic to $\Ass^*(\calT_\calR^r(L))$ functorially with respect to pre-Lie algebra maps.
\end{proof}

This last argument raises a natural question as to whether it is possible to define a $\calC$-enriched version of the dendriform operad so that the operad $\PL_\calC$ acts on primitive elements in $\Ass^c$-$\Dend_\calC$-bialgebras. 

\subsection{Minimal species of generators of the operad \texorpdfstring{$\PL_\calC$}{PLC}}

According to Corollary \ref{cor:PLgen}, the operad $\PL_\calC$ is generated by $\calC$-corollas. Analysing the argument that proves this result, it is in fact easy to prove that if $\calX$ generates the Cauchy monoid $\calC^{\tsbullet}$, then the operad $\PL_\calC$ is generated by $\calX$-corollas. 

\begin{example}
It is well known that the operad $\PL=\PL_{\uCom^*}$ has two different sets of generators: it can be generated by one binary operation (the pre-Lie product) or by the so called ``symmetric braces'' \cite{MR2175956}. The symmetric braces are exactly all $\uCom^*$-corollas. However, if we regard the Hopf cooperad $\uCom^*$ as an algebra for the operad $\uCom$, it is isomorphic to the free algebra on the singleton species, and the binary generator of $\PL$ corresponds to the corolla with one root and one non-root vertex, with the root vertex decorated by the singleton species. 
\end{example}

In the case of the brace operad, we obtain its minimal set of generators that seems to have never been studied before.

\begin{proposition}
The operad $\Br$ is generated by $\Lie^*$-corollas.
\end{proposition}

\begin{proof}
We already saw that considering the operad $\PL_{\uAss^*}$ amounts to considering planar rooted trees; in this case the corollas are the classical braces. To determine the species of generators of the Cauchy monoid $(\uAss^*)^{\tsbullet}$, we invoke the dual of the Cartier--Milnor--Moore theorem which easily implies that $(\uAss^*)^{\tsbullet}$ is isomorphic to the free commutative Cauchy monoid generated by the species~$\Lie^*$. 
\end{proof}

It would be interesting to study the combinatorics of this presentation, as well as minimal presentations of the operad $\PL_\calC$ for other choices of $\calC$. It is also reasonable to try and describe an analogue of the brace operad that acts on the Harrison complex of a commutative associative algebra. The Koszul dual cooperad of the commutative operad is $\Lie^*$ which does not have a Hopf structure, so the approach of \cite{MR3411136} is not directly applicable in this case.      

\subsection{Relationship to the twisting procedure}

Every operad $\PL_\calC$ receives the unit map from the operad $\PL_{\uCom^*}=\PL$, and therefore a map from the operad $\Lie$; therefore, as pointed in \cite{MR3411136}, to each such operad one may apply the construction of operadic twisting \cite{MR3299688}. It would be interesting to determine which of the operads $\PL_\calC$ have interesting homotopical properties with respect to operadic twisting, for example, for which of them one has
 \[
H_0(\Tw(\PL_\calC),d_{\Tw})\cong\Lie ,
 \]
generalising the existing results for the pre-Lie and the brace operad, see \cite{dotsenko2020homotopical}.

\section*{Acknowledgements} The first author is grateful to Mario Gon\c{c}alves Lamas for questions about brace algebras that prompted this note, and to Bruno Vallette and Pedro Tamaroff for useful comments. We thank the anonymous referee for thorough reading of the paper which made us realise that the original version of the main theorem was stated in higher generality than it actually holds. 

\bibliographystyle{plain} 
\bibliography{BracePBW.bib}

\end{document}